\documentclass[12pt]{amsart}
\usepackage{amssymb,amsmath,amsthm,bm}
\usepackage[colorinlistoftodos,prependcaption,textsize=tiny]{todonotes}
 \definecolor{darkblue}{RGB}{0,0,160}
\usepackage{fouriernc} 
\usepackage[colorlinks=true,allcolors=darkblue]{hyperref}
\DeclareSymbolFont{usualmathcal}{OMS}{cmsy}{m}{n}
\DeclareSymbolFontAlphabet{\mathcal}{usualmathcal}

\usepackage[T1]{fontenc}

\usepackage{color}

\usepackage{parskip}
\usepackage{setspace}

\usepackage{amsmath,amsthm,amscd,amssymb, mathrsfs}

\usepackage{latexsym}

\numberwithin{equation}{section}

\theoremstyle{plain}
\newtheorem{theorem}{Theorem}[section]
\newtheorem{lemma}[theorem]{Lemma}

\newtheorem{proposition}[theorem]{Proposition}

\theoremstyle{definition}
\newtheorem{definition}[theorem]{Definition}

\newtheorem{problem}[theorem]{Problem}

\theoremstyle{remark}

\newtheorem{case[theorem]}{Case}

\title[\parbox{14cm}{\centering{Averages and maximal averages  \hspace{1in}}} \quad]{Averages and maximal averages over Product $j$-varieties in finite fields}
\author{ Doowon Koh and Sujin Lee}

\address{Department of Mathematics\\
Chungbuk National University \\
Cheongju, Chungbuk 28644 Korea}
\email{koh131@chungbuk.ac.kr}

\address{Department of Mathematics\\
Chungbuk National University \\
Cheongju, Chungbuk 28644 Korea}
\email{sujin4432@chungbuk.ac.kr}

\thanks{
The first listed author was supported by Basic Science Research Program through the National Research Foundation of Korea(NRF) funded by the Ministry of Education, Science and Technology(NRF-2018R1D1A1B07044469)
}


\begin{document}

\begin{abstract}
We study both averaging  and maximal averaging problems for Product $j$-varieties defined by
$\Pi_j=\{x\in \mathbb F_q^d: \prod_{k=1}^d x_k=j\}$ for $j\in \mathbb F_q^*,$
where $\mathbb F_q^d$ denotes a $d$-dimensional vector space over the finite field $\mathbb F_q$ with $q$ elements. We prove the sharp $L^p\to L^r$ boundedness of averaging operators associated to  Product $j$-varieties.
We also obtain the optimal $L^p$ estimate for a maximal averaging operator related to a family of Product $j$-varieties $\{\Pi_j\}_{j\in \mathbb F_q^*}.$
\end{abstract}

\maketitle
\section{Introduction}

In recent year,  problems in Euclidean Harmonic analysis have been studied in the finite field setting, where the Euclidean structure is replaced by that of a vector space over a finite field.  This approach may be efficient to relate analysis problems to well-studied problems in  other areas such as the number theory and additive combinatorics. Moreover, problems in finite fields give us unique, interesting points as well as difficulties inherent in them.


In 1996, Wolff \cite{Wo98} suggested the finite field analogue of the Kakeya conjecture.
In 2008,  Dvir \cite{Dv09} solved this conjecture by using the polynomial method which is based on work in computer science. It has been applied to the Euclidean problems (for example, see \cite{GK15, Gu16, Gu18}).
In 2002, Mockenhaupt and Tao \cite{MT04} initially posed and studied the finite field restriction problem for algebraic varieties. Much attention has been given to this problem, in part because there exist some different restriction phenomena between the Euclidean problem and its finite field analog. We refer readers to \cite{Gr, KS12, LL12, KK14, Le15, Ko16, IKL17, RS18, IKSSP18, KPV18} for background and recent development on the finite field restriction estimates for algebraic varieties. For the setting of rings of integers, see \cite{HW18, HW19}. 


More recently,  Carbery, Stones, and Wright \cite{CSW08} introduced further harmonic analysis problems in the finite field setting. Among other things, they provided sharp results on the finite field (maximal) averaging problems associated to certain $k$-dimensional varieties generated by vector-valued polynomials. Stones \cite{St05} addressed a sharp maximal theorem for dilation of quadratic surfaces in finite fields. 


In this paper we will extend their work to more complicated varieties than quadratic surfaces. To precisely state our results, we need some notation and the definitions of the finite field averaging and  maximal averaging problems. Let $\mathbb F_q^d$ be a $d$-dimensional vector space over a finite field with $q$ elements. Throughout this paper, we assume that $q$ is an odd prime power. Namely, the characteristic of $\mathbb F_q$ is strictly greater than two. We endow the space $\mathbb F_q^d$ with normalized counting measure, denoted by $dx$, which satisfies
$$ \int_{x\in \mathbb F_q^d} f(x)~dx = q^{-d} \sum_{x\in \mathbb F_q^d} f(x),$$
where $f$ is a complex valued function on $\mathbb F_q^d.$
For $1\le s <\infty,$ we define
$$ \|f\|_{L^{s}(\mathbb F_q^d)} := \left(q^{-d}\sum_{x\in \mathbb F_q^d} |f(x)|^s\right)^{1/s}$$
and $\|f\|_{L^{\infty}(\mathbb F_q^d)} := \max_{x\in \mathbb F_q^d} |f(x)|.$

 Let $V \subset \mathbb F_q^d$ be an algebraic variety,  a set of common solutions to polynomial equations.   Normalized surface measure on $V$, denoted by $d\sigma$, is associated to the variety $V.$ Recall that 
 the surface measure $d\sigma$ is defined by the relation
$$ \int_{\mathbb F_q^d} f(x)~d\sigma(x) = \frac{1}{|V|} \sum_{x\in V} f(x),$$
where $|V|$ denotes the cardinality of $V.$

The convolution function $f\ast g$ of functions $f, g$ on $\mathbb F_q^d$ is defined by
$$ f\ast g(x):= \frac{1}{q^d} \sum_{y\in \mathbb F_q^d} f(x-y)g(y).$$
Taking $g=d\sigma$, we see that
$$f\ast d\sigma(x)= \frac{1}{|V|} \sum_{y\in V} f(x-y).$$
With the notation above, the averaging operator $A_V$ associated to $V$ is defined by
$$ A_Vf(x):=f\ast d\sigma(x),$$
where both $f$ and $A_Vf$ are complex-valued functions defined on $\mathbb F_q^d.$ 

The following averaging problem for $V$ is first posed by Carbery-Stones-Wright \cite{CSW08}.
\begin{problem} [Averaging Problem]
Let $d\sigma$ denote  normalized surface measure on an algebraic variety $V$ in $\mathbb F_q^d.$
Find all exponents $1\le p,r\le \infty$ such that  for some constant $C$ depending only on $p, r, d,$ and $V,$ the inequality
\begin{equation}\label{ADef} \|f\ast d\sigma\|_{L^r(\mathbb F_q^d)} \le C \|f\|_{L^p(\mathbb F_q^d)},\end{equation}
holds for all functions $f$ on $\mathbb F_q^d.$ 
\end{problem}
The most important condition in the finite field averaging problem is that  the operator norm is independent of the size of the underlying finite field $\mathbb F_q.$ 
We will write $A_V(p\to r)\lesssim 1$ if the averaging estimate \eqref{ADef} holds.

For each $1\le k \le d-1,$ Carbery-Stones-Wright \cite{CSW08} provided  concrete  $k$-dimensional surfaces in $\mathbb F_q^d$ for which they obtained the optimal result on the averaging  problem. Indeed, for $1\le k\le d-1,$ they considered a variety $V_k$ as the image of the  polynomial map $P_k:\mathbb F_q^k \to \mathbb F_q^d$ defined by
\begin{equation}\label{defVk} P_k(t)=(t_1,t_2, \ldots, t_k, t_1^2+\cdots+t_k^2, t_1^3+\cdots+t_k^3, \dots,  t_1^{d-k+1}+\cdots+t_k^{d-k+1} ),\end{equation}
and proved that $A_{V_k} (p\to r)\lesssim 1 $  if and only if $(1/p, 1/r)$ is contained in the convex hull of points $(0,0), (0,1), (1.1),$ and $(\frac{d}{2d-k}, \frac{d-k}{2d-k} ).$
 We notice that the following Fourier decay estimate of the surface measure $d\sigma_k$ on $V_k$ was one of the most important ingredients in proving the optimal averaging estimate related to $V_k$:
\begin{equation} \label{FDE}
\max_{m\in \mathbb F_q^d\setminus \{(0,\ldots, 0)\}} \left| (d\sigma_k)^{\vee} (m) \right|:=\max_{m\in \mathbb F_q^d\setminus \{(0,\ldots, 0)\}}\left| \frac{1}{q^k} \sum_{x\in V_k} \chi(m\cdot x) \right|  \le C q^{-\frac{k}{2}},
\end{equation}
where $C$ is independent of $q,$  and $\chi$ denotes a nontrivial additive character of $\mathbb F_q.$ 


In \cite{CSW08}, Carbery-Stones-Wright also posed the  maximal averaging problem for a family of algebraic varieties in $\mathbb F_q^d.$  Let $\mathcal{A}$ be an indexing set. For each $\alpha \in \mathcal{A}$, let $d\sigma_\alpha $ denote  normalized surface measure on an algebraic variety $V_\alpha$ in $\mathbb{F}_q^d$.
 Given any function $f : \mathbb{F}_q^d\rightarrow \mathbb{C},$ the maximal averaging operator $M$ is defined by
  \begin{equation*}
    Mf(x) := \sup_{\alpha \in \mathcal{A}}\left|f \ast d\sigma_\alpha(x)\right| \quad \textmd{for}\,\, x \in \mathbb F_q^d. \end{equation*}

\begin{problem}[Maximal Averaging Problem]
  Find all exponents $1 \leq p \leq \infty $ such that  the  inequality
 \begin{equation*}
   \|Mf\|_{L^p(\mathbb{F}_q^d)} \le C \|f\|_{L^p(\mathbb{F}_q^d)}
 \end{equation*}
holds for all complex-valued functions $f$ on $\mathbb F_q^d,$ where  the constant $C$ is independent of $q.$
\end{problem}
Carbery-Stones-Wright  \cite{CSW08} introduced  a family of varieties in $\mathbb F_q^d$ for which they deduced the optimal result on the maximal averaging problem. More precisely, they first considered an indexing set $\mathcal{A}$ with $|\mathcal{A}|= q^r$ for some $0\le r\le d-k.$  For each $\alpha\in \mathcal{A},$ letting $M_\alpha$ be an invertible $d\times d$ matrix over $\mathbb F_q$ and letting $b_\alpha$ be a vector in $\mathbb F_q^d,$ they considered the following $k$-dimensional variety
$$ V_{k,\alpha}:= \{M_\alpha x+ b_\alpha\in \mathbb F_q^d:  x\in V_k\},$$
where $V_k$ denotes the variety defined as in \eqref{defVk} and we identify a vector $x\in \mathbb F_q^d$ with  a $d\times 1$ matrix.
With the notation above, Carbery-Stones-Wright  \cite{CSW08} proved that
$$ \left\|\sup_{\alpha\in \mathcal{A}} |f\ast d\sigma_{k,\alpha}|\right\|_{L^p(\mathbb F_q^d)} \le C \|f\|_{L^p(\mathbb F_q^d)}$$
if and only if  $r\le k$ and $ p\ge \frac{r+k}{k},$
where $d\sigma_{k,\alpha}$ denotes the normalized surface measure on the variety $V_{k,\alpha}.$ Like the averaging problem for $V_k$,  the Fourier decay estimate \eqref{FDE} was mainly used in deducing the optimal result on the maximal averaging problem. 


The main purpose of this paper is to provide a complete solution of the averaging and maximal averaging problems related to certain varieties for which the Fourier decay estimate \eqref{FDE} is not satisfied.
To this end, for each $j\in \mathbb F_q^*,$ we consider an algebraic variety $\Pi_j$ in $\mathbb F_q^d$ defined by
$$ \Pi_j:=\{x=(x_1,x_2,\ldots, x_d)\in \mathbb F_q^d: \prod_{k=1}^d x_k=j\}.$$
We will call the variety $\Pi_j$ as Product $j$-variety. For each $j\in \mathbb F_q^*,$ let $d\mu_j$ be normalized surface measure on Product $j$-variety $\Pi_j.$ Unlike the Fourier decay estimate given in \eqref{FDE}, the bound of  $|(d\mu_j)^\vee(m)|$ becomes worse whenever we take any vector $m$ such that the number of zero components of $m=(m_1, m_2, \ldots, m_d)$ is large. 
\begin{definition} 
For each $m\in \mathbb F_q^d,$ we denote by $\ell_m$ the number of zero components of $m.$
\end{definition}
In fact, the inverse Fourier transform of the surface measure was explicitly computed as follows.
\begin{lemma} [\cite{CKP}, Lemma 3.1] \label{Decaymu} For each $j\in \mathbb F_q^*,$  let $d\mu_j$ denote  normalized surface measure on Product $j$-variety $\Pi_j$ in $\mathbb F_q^d.$   Then we have
\begin{equation} \label{conc1} (d\mu_j)^\vee (m)= (-1)^{d-\ell_m} ~(q-1)^{-(d-\ell_m)} \quad\mbox{if}\quad 1\le \ell_m \le d.\end{equation}
In addition, if $\ell_m =0$, then $|(d\mu_j)^\vee (m)|\lesssim q^{-\frac{(d-1)}{2}}.$
\end{lemma}
We may consider Product $j$-variety $\Pi_j$ in $\mathbb F_q^d$ as a $(d-1)$-dimensional surface since $|\Pi_j|\sim q^{d-1}.$ Notice that Lemma \ref{Decaymu} implies that for every $j\ne 0,$
$$ \max_{m\in \mathbb F_q^d\setminus \{(0,\ldots, 0)\}} \left| (d\mu_j)^{\vee} (m) \right| \sim q^{-1}.$$
Comparing this estimate with \eqref{FDE}, we see that if $d>3$, then
the Fourier decay on Product $j$-variety $\Pi_j$ is  much worse  than that on $V_{d-1}.$ For this reason, the same argument used by Carbery-Stones-Wright  \cite{CSW08} may not give us optimal results on both the averaging and maximal averaging problem related to Product $j$-varieties $\Pi_j.$
However, our first result below indicates that  even though the Fourier decay estimate on Product $j$-variety $\Pi_j$ is worse than that on the $V_{d-1},$  both varieties have the same mapping properties of the averaging operators.

\begin{theorem}\label{mainthm1} For each $j\in \mathbb F_q^*,$ let $d\mu_j$ denote the normalized surface measure on Product $j$-variety $\Pi_j.$  Then
$ A_{\Pi_j} (p\to r) \lesssim 1$ if and only if  $(1/p, 1/r)$ lies on the convex hull of points $(0,0), (0,1), (1.1),$ and $(\frac{d}{d+1}, \frac{1}{d+1} ).$
\end{theorem}
Our next result is related to maximal averages associated to a family $\{\Pi_j\}_{j\in \mathbb F_q^*}$ of  Product $j$-varieties $\Pi_j.$
\begin{theorem}\label{mainthm2} For each $j\in \mathbb F_q^*,$ let $d\mu_j$ denote the normalized surface measure on Product $j$-variety $\Pi_j.$ Then there is a constant $C$ independent of $q$ such that
$$\left\|\sup_{j\in \mathbb F_q^*} |f\ast d\mu_j|\right\|_{L^p(\mathbb F_q^d)} \le C \|f\|_{L^p(\mathbb F_q^d)}$$
if and only if  $p\ge \frac{d}{d-1}.$
\end{theorem}
\section{Necessary conditions}
In this section, we prove the necessary conditions for the averaging and maximal averaging estimates given in Theorem \ref{mainthm1} and Theorem \ref{mainthm2}, respectively.
We begin by proving the following necessary conditions for the boundedness of the averaging operator $A_{\Pi_j}$ associated with Product $j$-variety $\Pi_j.$

\begin{proposition}\label{Prop1} For each $j\in \mathbb F_q^*,$ let $d\mu_j$ denote the normalized surface measure on Product $j$-variety $\Pi_j.$ Suppose that the following inequality
\begin{equation}\label{AveN1} \|f\ast d\mu_j\|_{L^r(\mathbb F_q^d)} \lesssim \|f\|_{L^p(\mathbb F_q^d)}\end{equation}
holds for all functions $f$ on $\mathbb F_q^d.$ Then $(1/p, 1/r)$ is contained in the convex hull of points
$(0,0), (0,1), (1.1),$ and $(\frac{d}{d+1}, \frac{1}{d+1}).$
\end{proposition}
\begin{proof} We test the inequality \eqref{AveN1} with $f=\delta_{\mathbf{0}},$
where $\delta_{\mathbf{0}}(x)=1$ if $x=(0,\ldots, 0)$, and $0$ otherwise.
Taking $f=\delta_{\mathbf{0}},$ we have
\begin{equation*}
  \|f\|_{L^p(\mathbb{F}_q^d)}=\left(\frac{1}{q^d}\sum_{x\in\mathbb{F}_q^d}|\delta_{\textbf{0}}(x)|^p\right)^{\frac{1}{p}}=q^{-\frac{d}{p}}
\end{equation*}
and
\begin{align*}
 \|f\ast d\mu_j\|_{L^r(\mathbb{F}_q^d)} &= \left(\frac{1}{q^d}\sum_{x\in\mathbb{F}_q^d}|(\delta_{\textbf{0}}\ast d\mu_j)(x)|^r\right)^{\frac{1}{r}} \\
       & =\left(\frac{1}{q^d}\sum_{x\in \Pi_j}|\Pi_j|^{-r}\right)^{\frac{1}{r}}\sim q^{-\frac{d}{r}+\frac{(d-1)(1-r)}{r}}. \end{align*}
Hence, invoking the assumption \eqref{AveN1}  we obtain a necessary condition:
\begin{equation} \label{NecK1}
  \frac{1}{r}\geq \frac{d}{p}-d+1.
\end{equation}
Since the averaging operator is self-adjoint, we also have
$$ \frac{1}{p'}\geq \frac{d}{r'}-d+1$$
which is equivalent to
\begin{equation} \label{NecK2}
\frac{1}{r}\geq\frac{1}{dp}.
 \end{equation}
Since $1\le p, r\le \infty$,  the proposition follows from \eqref{NecK1} and \eqref{NecK2}.
\end{proof}
We now state and prove the necessary conditions for the boundedness of the maximal averaging operator given in Theorem \ref{mainthm2}.
\begin{proposition}
Assume that the following maximal averaging estimate for a family of Product $j$-varieties$\Pi_j$ holds for all functions $f$ on $\mathbb F_q^d:$
\begin{equation}\label{MAveN1}\left\|\sup_{j\in \mathbb F_q^*} |f\ast d\mu_j|\right\|_{L^p(\mathbb F_q^d)} \lesssim \|f\|_{L^p(\mathbb F_q^d)}.\end{equation}
Then we have
$$ p\ge \frac{d}{d-1}.$$
\end{proposition}
\begin{proof}
As in the proof of Proposition \ref{Prop1}, we test the inequality \eqref{MAveN1} with $f=\delta_{\mathbf{0}}.$ Then it follows that
$$ \|f\|_{L^p(\mathbb{F}_q^d)}=q^{-\frac{d}{p}}.$$
On the other hand, we have
$$ \left\|\sup_{j\in \mathbb F_q^*} |f\ast d\mu_j|\right\|_{L^p(\mathbb F_q^d)}\ge \left(\frac{1}{q^d}\sum_{x\in \bigcup\limits_{j\in \mathbb F_q^*} \Pi_j} \left(\sup_{j\in \mathbb F_q^*}|(\delta_{\textbf{0}}\ast d\mu_j)(x)|\right)^p\right)^{\frac{1}{p}} $$
Since $(\delta_{\mathbf{0}}\ast d\mu_j)(x)= \frac{1}{|\Pi_j|} 1_{\Pi_j}(x) \sim q^{-(d-1)} 1_{\Pi_j}(x),$ we have
$$\left\|\sup_{j\in \mathbb F_q^*} |f\ast d\mu_j|\right\|_{L^p(\mathbb F_q^d)}\gg q^{-\frac{d}{p}} q^{-d+1} \left(\sum_{x\in \bigcup\limits_{j\in \mathbb F_q^*} \Pi_j} \left(\sup_{j\in \mathbb F_q^*} 1_{\Pi_j}(x)\right)^p\right)^{\frac{1}{p}} \sim q^{-\frac{d}{p}} q^{-d+1} q^{\frac{d}{p}} = q^{-d+1}.$$
Thus, the proposition follows from the hypothesis \eqref{MAveN1}.
\end{proof}

\section{Sufficient conditions}

\begin{theorem} \label{SuffA}
If $(1/p, 1/r)$ is contained in the convex hull of points $(0,0), (0,1), (1,1)$ and $ (\frac{d}{d+1}, \frac{1}{d+1}),$ then the averaging inequality
\begin{equation}\label{AES} \|f\ast d\mu_j\|_{L^r(\mathbb F_q^d)} \lesssim \|f\|_{L^p(\mathbb F_q^d)}\end{equation}
holds for all functions $f$ on $\mathbb F_q^d$ and all $j\ne 0.$
\end{theorem}

\begin{proof}
Since $(f\ast d\mu_j)(x)= |\Pi_j|^{-1} \sum_{x\in \Pi_j} f(x)$ for all $x\in \mathbb F_q^d,$ it is clear that the inequality \eqref{AES} holds in the case when $(1/p, 1/r)=(0,0).$ By a direct computation,  it is also true for $(1/p, 1/r)=(1,1).$ Thus, by using an interpolation theorem and nesting property of norms,  we only need to verify
the inequality \eqref{AES} for $p=\frac{d+1}{d}$ and $r=d+1.$ In other words, we aim to prove that the averaging estimate
\begin{equation} \label{MA}\|f\ast d\mu_j\|_{L^{d+1}(\mathbb F_q^d)} \lesssim \|f\|_{L^{\frac{d+1}{d} }(\mathbb F_q^d)}\end{equation}
holds for all function $f$ on $\mathbb F_q^d.$ For each $m\in \mathbb F_q^d$, let $\ell_m$ be the number of zero components of the vector $m$. Now, for each $k=0, 1, \ldots, d,$ we define
\begin{equation*} N_k:=\{m\in \mathbb F_q^d: \ell_m=k\}.\end{equation*}
It is obvious that $\mathbb F_q^d = \bigcup\limits_{k=0}^d N_k.$
Since ${(d\mu_j)}^\vee =  1_{N_0}(d\mu_j)^\vee  + \sum\limits_{k=1}^d 1_{N_k} (d\mu_j)^\vee ,$
we can decompose $d\mu_j$ as
\begin{equation}\label{insert1} d\mu_j = \widehat{1_{N_0}} \ast d\mu_j + \sum_{k=1}^d \widehat{1_{N_k}} \ast d\mu_j.\end{equation}
By the definition of $N_k$ and  the first part \eqref{conc1} of Lemma \ref{Decaymu}, we see that for $k=1,2, \ldots, d,$
$$ 1_{N_k} (d\mu_j)^\vee = (-1)^{d-k} (q-1)^{-d+k} 1_{N_k},$$
which in turn gives us
$$ \widehat{1_{N_k}} \ast d\mu_j =  (-1)^{d-k} (q-1)^{-d+k} \widehat{1_{N_k}}.$$
This can be combined with \eqref{insert1} to see that
\begin{equation}\label{defmuj}  d\mu_j = \widehat{1_{N_0}} \ast d\mu_j + \sum_{k=1}^d  (-1)^{d-k} (q-1)^{-d+k} \widehat{1_{N_k}}.\end{equation}
For each $j\ne 0,$ let $\Omega_j=  \widehat{1_{N_0}} \ast d\mu_j.$ We have
$$\|f\ast d\mu_j\|_{L^{d+1}(\mathbb F_q^d)} \le \|f\ast \Omega_j\|_{L^{d+1}(\mathbb F_q^d)} +\sum_{k=1}^d   (q-1)^{-d+k} \|f\ast \widehat{1_{N_k}} \|_{L^{d+1}(\mathbb F_q^d)}.$$
Hence, to prove \eqref{MA}, it will be enough to verify the following estimates:
\begin{equation}\label{MA21} \|f\ast \Omega_j\|_{L^{d+1}(\mathbb F_q^d)} \lesssim \|f\|_{L^{\frac{d+1}{d} }(\mathbb F_q^d)},
\end{equation}
and for every $k=1,2,\ldots, d,$
\begin{equation} \label{MA22}
(q-1)^{-d+k} \|f\ast \widehat{1_{N_k}} \|_{L^{d+1}(\mathbb F_q^d)}\lesssim \|f\|_{L^{\frac{d+1}{d} }(\mathbb F_q^d)}.
 \end{equation}

\subsection{proof of the inequality \eqref{MA21}} We notice that the inequality \eqref{MA21} can be obtained by interpolating the following two estimates:
\begin{equation}\label{twotwo}
 \|f\ast \Omega_j\|_{L^2(\mathbb{F}_q^d)} \lesssim q^{-\frac{d-1}{2}}\|f\|_{L^2(\mathbb{F}_q^d)},
 \end{equation}
 \begin{equation}\label{infty1}
   \|f\ast \Omega_j\|_{L^{\infty}(\mathbb{F}_q^d)} \lesssim q\|f\|_{L^1(\mathbb{F}_q^d)}.
  \end{equation}
Thus, we only need to show that the inequalities \eqref{infty1} and \eqref{twotwo}  hold for any functions $f$ on $\mathbb F_q^d.$
We can easily verify the inequality \eqref{twotwo} by applying the Plancherel theorem and the second conclusion of  Lemma \ref{Decaymu} as follows:
\begin{align*}
 \|f\ast \Omega_j\|_{L^2(\mathbb{F}_q^d)} &=\| f^\vee \Omega_j^\vee\|_{\ell^2(\mathbb{F}_q^d)}\\
         &=\left(\sum_{m\in N_0}|f^\vee(m)|^2|(d\mu_j)^{\vee}(m)|^2\right)^{\frac{1}{2}}\\
         & \lesssim q^{-\frac{d-1}{2}}\left(\sum_{m\in N_0}|\widehat{f}(m)|^2\right)^{\frac{1}{2}} \\
          & \lesssim q^{-\frac{d-1}{2}}\|\widehat{f}\|_{\ell^2(\mathbb{F}_q^d)}
         = q^{-\frac{d-1}{2}}\|f\|_{L^2(\mathbb{F}_q^d)}.
        \end{align*}
To prove the inequality \eqref{infty1}, we first notice by Young's inequality for convolution functions that
$$ \|f\ast \Omega_j\|_{L^{\infty}(\mathbb{F}_q^d)} \le \|\widehat{1_{N_0}} \ast d\mu_j\|_{L^\infty(\mathbb F_q^d)} ~\|f\|_{L^1(\mathbb F_q^d)}.$$
Then the inequality \eqref{infty1} will be established by showing that the inequality
\begin{equation}\label{maxeq} \max_{x\in \mathbb F_q^d} |(\widehat{1_{N_0}} \ast d\mu_j)(x)|\lesssim q
\end{equation}
holds for all $j\in \mathbb F_q^*.$
To prove this inequality, we fix $x\in \mathbb F_q^d, j\in \mathbb F_q^*$, and observe that
$$ |(\widehat{1_{N_0}} \ast d\mu_j)(x)|= \left|\frac{1}{|\Pi_j|} \sum_{y\in \Pi_j} \widehat{1_{N_0}}(x-y)\right| = \left| \frac{1}{|\Pi_j|} \sum_{y\in \Pi_j} \sum_{m\in N_0} \chi(m\cdot (y-x))\right|$$
$$\le |\Pi_j|^{-1} \sum_{y\in \mathbb F_q^d} \left| \sum_{m\in N_0} \chi(m\cdot (y-x))\right|
= |\Pi_j|^{-1}  \sum_{z\in \mathbb F_q^d}  \left| \sum_{m\in N_0} \chi(m\cdot z)\right|$$
$$\sim q^{-d+1} \sum_{k=0}^d \sum_{y\in N_k} \left| \sum_{m_1, m_2,\ldots, m_d\in \mathbb F_q^*} \chi(m\cdot z)\right|.$$
By using the orthogonality of $\chi$ and the definition of $N_k$, we conclude
$$|(\widehat{1_{N_0}} \ast d\mu_j)(x)| \lesssim q^{-d+1} \sum_{k=0}^d |N_k| (q-1)^k \sim q,$$
where the above similarity follows from the fact that  $|N_k|\sim q^{d-k}.$
This proves the inequality \eqref{maxeq}, as required. We have finished the proof of the inequality \eqref{MA21}.

\subsection{proof of the inequality \eqref{MA22}}
By Young's inequality for convolution functions, we have
$$ (q-1)^{-d+k} \|f\ast \widehat{1_{N_k}} \|_{L^{d+1}(\mathbb F_q^d)} \le q^{-d+k}\|f\|_{L^{\frac{d+1}{d}}(\mathbb{F}_q^d)}\|\widehat{1_{N_k}}\|_{L^{\frac{d+1}{2}}(\mathbb{F}_q^d)}.$$
Hence, to prove the  inequality \eqref{MA22}, it suffices to show that for each $k=1,2, \ldots, d,$
\begin{equation*}\label{AAk1}
\|\widehat{1_{N_k}}\|_{L^{\frac{d+1}{2}}(\mathbb{F}_q^d)} \lesssim q^{d-k}.
\end{equation*}
We will prove this inequality  separately in the cases of  $d=2$ and $d\ge 3.$

\textbf{Case 1:} Let $d\ge 3.$ 
Since $ 2\le  (d+1)/2 <\infty$ for $d\ge 3,$ we can invoke the Hausdorff-Young inequality to deduce the required estimate. More precisely, we have
$$\|\widehat{1_{N_k}}\|_{L^{\frac{d+1}{2}}(\mathbb{F}_q^d)} \le  \|1_{N_k}\|_{\ell^{\frac{d+1}{d-1}}(\mathbb{F}_q^d)} = |N_k|^{\frac{d-1}{d+1}} \sim q^{(d-k) \frac{d-1}{d+1}} \le q^{d-k},$$
as desired.

\textbf{Case 2:}  Let $d=2.$ We aim to prove that for $k=1,2,$ 
$$ \|\widehat{1_{N_k}}\|_{L^{\frac{3}{2}}(\mathbb{F}_q^2)} \lesssim q^{2-k}.$$
For  $k=2,$  it is clear that $N_2=\{(0,0)\}.$  Hence,  the above inequality  follows  by observing  $\widehat{1_{N_2}}(x)=1 $ for all $x\in \mathbb F_q^2.$ 
To prove the above inequality for the case of $k=1,$  we first notice that 
$$N_1=\left(\mathbb F_q^*\times \{0\} \right) \cup \left( \{0\}\times \mathbb F_q^*\right)$$
which implies that $|N_1|= 2(q-1).$
For any $x\in \mathbb F_q^2,$  we have
$$ |\widehat{N_1}(x)| \le |N_1|\sim q.$$
Therefore, it follows that 
$$ \|\widehat{1_{N_1}}\|_{L^{\frac{3}{2}}(\mathbb{F}_q^2)} \le \max_{x\in \mathbb F_q^2}  |\widehat{1_{N_1}}(x)| \sim q.$$
This completes the proof of  the inequality  \eqref{MA22}.
\end{proof}

\begin{theorem}\label{SuffMA} If  $p\ge \frac{d}{d-1}$, then the following maximal averaging estimate for a family $\{\Pi_j\}_{j\in \mathbb F_q^*}$ of Product $j$-varieties $\Pi_j$ holds for all functions $f$ on $\mathbb F_q^d:$
\begin{equation}\label{MASuff}\left\|\sup_{j\in \mathbb F_q^*} |f\ast d\mu_j|\right\|_{L^p(\mathbb F_q^d)} \lesssim \|f\|_{L^p(\mathbb F_q^d)}.\end{equation}
\end{theorem}
\begin{proof} It is not hard to check that for every $x\in \mathbb F_q^d,$
$$ \sup_{j\in \mathbb F_q^*} |(f\ast d\mu_j)(x)| \le \max_{y\in \mathbb F_q^d} |f(y)|.$$
Hence, the inequality \eqref{MASuff} is true for $p=\infty.$  By interpolation theorem, it therefore suffices to prove the inequality \eqref{MASuff} for $p=\frac{d}{d-1}.$ Namely, our task is to verify the following estimate:
\begin{equation}\label{ourgoal}\left\|\sup_{j\in \mathbb F_q^*} |f\ast d\mu_j|\right\|_{L^{\frac{d}{d-1}}(\mathbb F_q^d)} \lesssim \|f\|_{L^{\frac{d}{d-1}}(\mathbb F_q^d)}.\end{equation}
As in \eqref{defmuj}, for each $j\in \mathbb F_q^*,$ we can write
$$  d\mu_j = \Omega_j+ \sum_{k=1}^d  (-1)^{d-k} (q-1)^{-d+k} \widehat{1_{N_k}},$$
where $\Omega_j:=\widehat{1_{N_0}} \ast d\mu_j.$  Thus we have
$$\left\|\sup_{j\in \mathbb F_q^*} |f\ast d\mu_j|\right\|_{L^{\frac{d}{d-1}}(\mathbb F_q^d)} \le
\left\|\sup_{j\in \mathbb F_q^*} |f\ast \Omega_j|\right\|_{L^{\frac{d}{d-1}}(\mathbb F_q^d)}
+ \sum_{k=1}^d  (q-1)^{-d+k}\left\| f\ast \widehat{1_{N_k}} \right\|_{L^{\frac{d}{d-1}}(\mathbb F_q^d)}.$$
Therefore, to prove \eqref{ourgoal} (namely, to complete the proof of Theorem \ref{SuffMA}), it suffices to prove the following estimates:
\begin{equation}\label{MA31} \left\|\sup_{j\in \mathbb F_q^*} |f\ast \Omega_j|\right\|_{L^{\frac{d}{d-1}}(\mathbb F_q^d)} \lesssim \|f\|_{L^{\frac{d}{d-1}}(\mathbb F_q^d)},
\end{equation}
and for every $k=1,2,\ldots, d,$
\begin{equation} \label{MA32}
(q-1)^{-d+k}\left\| f\ast \widehat{1_{N_k}} \right\|_{L^{\frac{d}{d-1}}(\mathbb F_q^d)}\lesssim \|f\|_{L^{\frac{d}{d-1}}(\mathbb F_q^d)}.
 \end{equation}
First, let us verify the inequality  \eqref{MA32}.  Notice by the H\H{o}lder inequality that if $1\le t_1\le t_2\le \infty,$ then
$$ \|f\|_{L^{t_1}(\mathbb F_q^d)} \le \|f\|_{L^{t_2}(\mathbb F_q^d)}.$$
By this nesting property of norms, it is not hard to see that the inequality \eqref{MA32} follows from the inequality \eqref{MA22}.

Finally, we prove the inequality \eqref{MA31} which is a direct consequence from interpolating the following two estimates:
\begin{equation}\label{Inter11}
\left\|\sup_{j\in \mathbb F_q^*} |f\ast \Omega_j|\right\|_{L^{1}(\mathbb F_q^d)} \lesssim q \|f\|_{L^{1}(\mathbb F_q^d)},\end{equation}
\begin{equation}\label{Inter22} \left\|\sup_{j\in \mathbb F_q^*} |f\ast \Omega_j|\right\|_{L^{2}(\mathbb F_q^d)} \lesssim q^{\frac{2-d}{2}} \|f\|_{L^{2}(\mathbb F_q^d)}. \end{equation}
\end{proof}

Hence, to finish the proof, it remains to prove the inequalities \eqref{Inter11} and \eqref{Inter22}, which can be done by adapting an argument from \cite{CSW08}. The details are as follows.
The inequality \eqref{Inter11} follows, because we have
\begin{align*} \left\|\sup_{j\in \mathbb F_q^*} |f\ast \Omega_j|\right\|_{L^{1}(\mathbb F_q^d)}
&\le \left\| |f|\ast \sup_{j\in \mathbb F_q^*}|\Omega_j|\right\|_{L^{1}(\mathbb F_q^d)}\\
&\le \| f\|_{L^1(\mathbb F_q^d)} ~\left\|\sup_{j\in \mathbb F_q^*}|\Omega_j|\right\|_{L^1(\mathbb F_q^d)} \lesssim q \|f\|_{L^{1}(\mathbb F_q^d)},\end{align*}
where the last inequality $\lesssim$ is a direct consequence from the inequality \eqref{maxeq}.

For the inequality \eqref{Inter22}, we have
\begin{align*} \left\|\sup_{j\in \mathbb F_q^*} |f\ast \Omega_j|\right\|_{L^{2}(\mathbb F_q^d)}
&\le \left\|\left(\sum_{j\in \mathbb F_q^*} |f\ast \Omega_j|^2\right)^{\frac{1}{2}}\right\|_{L^{2}(\mathbb F_q^d)} =\left(\sum_{j\in \mathbb F_q^*} \|f\ast \Omega_j\|^2_{L^2(\mathbb F_q^d)}\right)^{\frac{1}{2}}\\
&\le \left( \sum_{j\in \mathbb F_q^*} \left(\max_{m\in \mathbb F_q^d} |\Omega_j^{\vee}(m)|^2 \right) \|f^\vee\|^2_{\ell^2(\mathbb F_q^d)} \right)^{\frac{1}{2}}\\
&\le \left(\max_{j\in \mathbb F_q^*, m\in \mathbb F_q^d} |\Omega_j^{\vee}(m)|\right) q^{\frac{1}{2}}  \|f\|_{L^2(\mathbb F_q^d)}.
 \end{align*}
Since $\Omega_j^{\vee} = 1_{N_0}~ (d\mu_j)^\vee,$   the second part of Lemma \ref{Decaymu} implies that
the maximum value above is dominated by $\sim q^{-\frac{d-1}{2}},$ and hence the inequality \eqref{Inter22} is obtained, as desired. We have finished the proof of Theorem \ref{mainthm2}.

\bibliographystyle{amsplain}

\begin{thebibliography}{10}


\bibitem{CKP} D. Cheong, D. Koh, and T. Pham, \textit{Extension theorems for Hamming varieties over finite fields,} preprint (2019), arXiv:1903.03904.


\bibitem{CSW08} A. Carbery, B. Stones, and J. Wright, \textit{Averages in vector spaces over finite fields}, Math. Proc. Cambridge Philos. Soc.  \textbf{144}  (2008),  no. 1, 13-27.

\bibitem{Dv09} Z. Dvir, \textit{On the size of Kakeya sets in finite fields}, J. Amer. Math. Soc.  \textbf{22}  (2009),  no. 4, 1093-1097.

\bibitem{Gr} B. Green, \emph{Restriction and Kakeya phenomena}, lecture note,
http://people.maths.ox.ac.uk/greenbj/papers/rkp.pdf. 

\bibitem{Gu16} L. Guth, \textit{A restriction estimate using polynomial partitioning}, J. Amer. Math. Soc. \textbf{29}  (2016),  no. 2, 371-413.

\bibitem{Gu18} L. Guth, \textit{Restriction estimates using polynomial partitioning II}, Acta Math.  \textbf{221}  (2018),  no. 1, 81-142.

\bibitem{GK15}  L. Guth and N.H. Katz, \textit{On the Erd\H{o}s distinct distances problem in the plane}, Ann. of Math. (2) \textbf{181}  (2015),  no. 1, 155-190.

\bibitem{HW18} J. Hickman and J. Wright, \textit{The Fourier restriction and Kakeya problems over rings of integers modulo N,}  Discrete Anal. 2018, Paper No. 11, 54 pp. 

\bibitem{HW19} J. Hickman and J. Wright, \textit{An abstract $L^2$ Fourier restriction theorem,} Math. Res. Lett. \textbf{26} (2019), no. 1, 75-100. 

\bibitem{IKL17} A. Iosevich, D. Koh, and M. Lewko, \emph{Finite field restriction estimates for the paraboloid in high even dimensions,}  J. Funct. Anal. \textbf{278} (2020), no. 11, 108450, 16 pp.

\bibitem{IKSSP18} A. Iosevich, D. Koh, L. Sujin, C. Shen, and T. Pham, \emph{On restriction estimates for the zero radius sphere over finite fields,} Canad. J. Math. Published online by Cambridge University Press: 27 February 2020, pp. 1-18.


\bibitem{Ko16} D. Koh, \emph{Conjecture and improved extension theorems for paraboloids in the finite field setting},  Math. Z. \textbf{294} (2020), no. 1-2, 51-69.

\bibitem{KK14} H. Kang and D. Koh, \textit{Weak version of restriction estimates for spheres and paraboloids in finite fields}, J. Math. Anal. Appl. \textbf{419}  (2014),  no. 2, 783-795.

\bibitem{KPV18} D. Koh, T. Pham, and L. A, Vinh, \emph{Extension theorems and a connection to the Erd\H{o}s-Falconer distance problem over finite fields}, J. Funct. Anal. accepted for publication (2021).


\bibitem{KS12} D. Koh and C. Shen, \emph{Sharp extension theorems and Falconer distance problems for algebraic curves in two dimensional vector spaces over finite fields}, Rev. Mat. Iberoam.  \textbf{28}  (2012),  no. 1, 157-178.


\bibitem{Le15}  M. Lewko, \emph{New restriction estimates for the 3-d paraboloid over finite fields},   Adv. Math. {\bf 270} (2015), no. 1, 457-479.



\bibitem{LL12} A. Lewko and M. Lewko, \emph{Endpoint restriction estimates for the paraboloid over finite fields}, Proc. Amer. Math. Soc. \textbf{140} (2012), 2013-2028.


\bibitem{MT04} G. Mockenhaupt, and T. Tao, \emph{Restriction and Kakeya phenomena for finite fields}, Duke Math. J. \textbf{121} (2004), 1, 35-74.

\bibitem{St05} B. Stones, \textit{Aspects of Harmonic Analysis over Finite Fields}, Thesis (Ph.D.)–University of Edinburgh (2005).

\bibitem{RS18} M. Rudnev and I. Shkredov, \emph{On the restriction problem for discrete paraboloid in lower dimension}, Adv. Math. {\bf 339} (2018), 657-671.


\bibitem{Wo98} T. Wolff, \textit{Recent work connected to the Kakeya problem,} In Prospects in Mathematics: Invited Talks on the Occasion of the 250th Anniversary of Princeton University, March 17-21, 1996, Page 129, AMS, pages 129–162, 1999.






\end{thebibliography}

\end{document}